\documentclass[12pt]{article}
\font\smallit=cmti10

\usepackage{amsmath,amsthm,amssymb,mathrsfs}
\usepackage{amsfonts}
\usepackage{latexsym}
\usepackage{amscd}

\usepackage{dsfont}
\usepackage{bbm}
\usepackage{rotating}
\usepackage{newlfont}
\usepackage{enumitem}
\usepackage[english]{babel}
\newtheorem{Main}{Theorem}[section]
\newtheorem{Lemma 1}[Main]{Lemma}
\newtheorem{Lemma 2}[Main]{Lemma}
\newtheorem{Lemma 3}[Main]{Lemma}
\newtheorem{Corollary}[Main]{Corollary}
\newtheorem{sub}[Main]{Theorem}
\newtheorem{exact}[Main]{Theorem}
\newtheorem{Coro}[Main]{Corollary}

\def\h25{\hspace{-.3cm}}

\begin{document}

\begin{center}
{\bf On unavoidable obstructions in Gaussian walks}
\vskip 20pt {\bf Sai Teja Somu\ \ and \ \ Ram Krishna Pandey}\\
{\smallit Department of Mathematics, Indian Institute of Technology
Roorkee, Roorkee - 247 667, India}\\
{\tt somuteja@gmail.com;\ ramkpandey@gmail.com }\\
\end{center}
\vskip 30pt

\centerline{\bf Abstract} In this paper we investigate a problem about certain walks in the ring of Gaussian integers. Let $n,d$ be two natural numbers. Does there exist a sequence of Gaussian integers $z_j$ such that $|z_{j+1}-z_j|=1$ and a pair of indices $r$  and $s$, such that $z_{r}-z_{s}=n$ and for all indices $t$ and $u$, $z_{t}-z_{u}\neq d$? If there exists such a sequence we call $n$ to be $d$ avoidable. Let $A_n$ be the set of all $d\in \mathbb{N}$ such that $n$ is not $d$ avoidable. Recently, Ledoan and Zaharescu proved that $\{d \in \mathbb{N} : d|n\}\subset A_n$. We extend this result by giving a necessary and sufficient condition for $d\in A_n$ which answers a question posed by Ledoan and Zaharescu. We also find a precise formula for the cardinality of $A_n$ and answer three other questions raised in the same paper.
%
\section{Introduction}
Walks in Gaussian integers have been investigated in the past by several authors (\cite{A}, \cite{B}, \cite{C}, \cite{D}, \cite{E}) to work on the question of whether one can start in the vicinity of the origin of the complex plane and walk to infinity using the Gaussian primes and only taking steps of bounded length. Recently, in \cite{F} there has been an investigation in a different direction. In that paper the authors have investigated walks of unit steps and demonstrated that there exists some kind of divisibility obstruction. Let $n,d$ be two natural numbers if there exists a sequence of Gaussian integers $z_j$ such that $|z_{j+1}-z_j|=1$ and a pair of indices $r$  and $s$, such that $z_{r}-z_{s}=n$ and for all indices $t$ and $u$, $z_{t}-z_{u}\neq d$.
 If there exists such a sequence we call $n$ to be $d$ avoidable. Let $A_n$ be the set of all $d\in \mathbb{N}$ such that $n$ is not $d$ avoidable. In \cite{F}, they prove that the set of divisors of $n$ is a subset of $A_n$. That is, if $d|n$ then $n$ is not $d$ avoidable.

In section 2, we give the precise structure of $A_n$ along with the cardinality of $A_n$. From this precise definition of $A_n$, we answer four of the six questions asked in \cite{F} in section 3.
Before going to the main theorem of section 2, let us consider the following example.

\noindent {\bf Example}. Let $n=20$. We consider three sequences $S_1$, $S_2$ and $S_3$ defined as follows:
\begin{align*}
S_1:~~ & z_0 = 0, z_{12} = 10,  z_{24} = 20,\\
& z_j=j-1-i \hspace{2 mm} \text{for} \hspace{2 mm} 1\leq j\leq 11,\\
& z_j=j-3+i \hspace{2 mm} \text{for} \hspace{2 mm} 13\leq j\leq 23.
\end{align*}
Here, $z_{24}-z_0 = 20$ and one can see that the set of all positive integer differences is $\{1,2,3,4,5,6,7,8,9,10,20\}$.

$S_2:~~  z_0=0,z_1=i,z_2=2i,z_3=1+2i,z_4=2+2i,z_5=3+2i,z_6=4+2i, z_7=5+2i,z_8=6+2i,z_9=7+2i,z_{10}=7+i,z_{11}=7,z_{12}=7-i, z_{13}=8-i,z_{14}=9-i,z_{15}=9,z_{16}=9+i,z_{17}=10+i,z_{18}=11+i, z_{19}=12+i,z_{20}=13+i,z_{21}=14+i,z_{22}=14,z_{23}=14-i,z_{24}=14-2i, z_{25}=15-2i,z_{26}=16-2i,z_{27}=17-2i,z_{28}=18-2i,z_{29}=19-2i,z_{30}=20-2i, z_{31}=20-i,z_{32}=20.$\\
The set of all positive differences is $\{1,2,3,4,5,6,7,9,10,11,12,13,14,20\}$.

$S_3:~~ z_0=0,z_1=i,z_2=2i,z_3=1+2i,z_4=2+2i,z_5=3+2i,z_6=4+2i, z_7=5+2i,z_8=6+2i,z_9=7+2i,z_{10}=8+2i,z_{11}=8+i,z_{12}=8, z_{13}=8-i,z_{14}=9-i,z_{15}=10-i,z_{16}=10,z_{17}=10+i,z_{18}=11+i, z_{19}=12+i,z_{20}=13+i,z_{21}=14+i,z_{22}=15+i,z_{23}=16+i,z_{24}=16, z_{25}=16-i,z_{26}=16-2i,z_{27}=17-2i,z_{28}=18-2i,z_{29}=19-2i,z_{30}=20-2i, z_{31}=20-i,z_{32}=20.$\\
The set of all positive differences is $\{1,2,3,4,5,6,7,8,10,11,12,13,14,15,16,20\}$.

The intersection of positive difference sets of $S_1$, $S_2$ and $S_3$ is $\{1,2,3,4,5,6,7,10,20\}$. If we try to go from $0$ to $20$ in any walk we suspect that we cannot avoid any number that belongs to the intersection. We believe that $A_{20}=\{1,2,3,4,5,6,7,10,20\}$.

Let $n = k(n,d)d + r(n,d),$ where $r(n,d)$ is a unique integer belonging to $\left[-\left\lfloor\frac{d}{2}\right\rfloor, \left\lceil\frac{d}{2}\right\rceil-1\right].$ Consider the following table
 \begin{center}
  \begin{tabular}{| l | l | l | l | l |}
  \hline
  $d$ & $k(20,d)$ & $r(20,d)$ & $A_{20}$ & $k(20,d)-|r(20,d)|$ \\
  \hline 1 & 20 & 0 & $\in$ & 20\\
  \hline 2 & 10 & 0 & $\in$ & 10\\
  \hline 3 & 7 & -1 & $\in$ & 6\\
  \hline 4 & 5 & 0 & $\in$ & 5\\
  \hline 5 & 4 & 0 & $\in$ & 4\\
  \hline 6 & 3 & 2 & $\in$ & 1\\
  \hline 7 & 3 & -1 & $\in$ & 2\\
  \hline 8 & 3 & -4 & $\notin$ & -1\\
  \hline 9 & 2 & 2 & $\notin$ & 0\\
  \hline 10 & 2 & 0 & $\in$ & 2\\
  \hline 11 & 2 & -2 & $\notin$ & 0\\
  \hline 12 & 2 & -4 & $\notin$ &-2\\
  \hline 13 & 2& -6& $\notin$ &-4\\
  \hline 14 & 1& 6& $\notin$ &-5\\
  \hline 15 & 1& 5&$\notin$ &-4\\
  \hline 16 & 1& 4&$\notin$ &-3\\
  \hline 17 & 1& 3&$\notin$ &-2\\
  \hline 18 & 1& 2&$\notin$ &-1\\
  \hline 19 & 1& 1&$\notin$ & 0\\
  \hline 20 & 1& 0&$\in$ & 1\\
  \hline
 \end{tabular}
 \end{center}
 If we assume for a moment that $A_{20}=\{1,2,3,4,5,6,7,10,20\},$ then from the table, we observe that $d\in A_{20}$ if and only if $k(20,d)-|r(20,d)|\geq 1$. We prove that this property is true not only for $20$ but for all natural numbers $n$. This is the main result of the paper which is presented in the next section.

 \section{Main results}

 \begin{Main} \label{main}
 Let $n\geq 1$ and $d\geq 1$ be integers. Then $d\in A_n$ if and only if
 $k(n,d)\geq |r(n,d)|+1.$
 \end{Main}

 In order to prove one part of the theorem we require the following lemma.

\begin{Lemma 1}\label{lemma1}
Let $n,$ $d>0$ and $n=kd+r$. If $k\geq |r|+1$ then $d\in A_n$.
\end{Lemma 1}
\begin{proof}
Let $r$ be the number with least absolute value such that there exists a $k\geq |r|+1$ and $n=kd+r$ for some $n>0$ and $d>0$ such that $d\notin A_n$. Then $r\neq 0$ as $r=0$ would imply that $d|n$ and from \cite{F} $d\in A_n$. Next $d\notin A_n$ implies that there exists a sequence $S = (z_j)$ of Gaussian integers such that $z_0=0\in S$ and $z_l=n\in S$ where $z_l$ is the final term and $|z_{p+1}-z_p|=1$ for $0\leq p\leq l-1$, $z_j-z_{j'}\neq d$ for $0\leq j,$ $j' \leq l$.

Now we create an  $l+2$ term sequence $S' = (z_p')$ with $z'_p=z_p$ for $0\leq p\leq l$ and $z'_{l+1}=z$ where
\[
z:=
\left\{
  \begin{array}{ll}
    n-1, & \hbox{if $r\geq 1$;} \\
    n+1, & \hbox{if $r \leq -1$,}
  \end{array}
\right.
\]
$z=kd+r'$ with $|r'|=|r|-1$. Hence, the minimality assumption implies that $d\in A_z$ since $k\geq |r'|+2$. Hence, there should exist two points $x$, $y\in S'$ such that $x-y=\pm d$ and  both $x$, $y$ cannot be in $S$ as $z_j-z_{j'}\neq d$ for $0\leq j,$ $j' \leq l$.
Without loss of generality, let $y = z_{l+1}'=z$ and $x=z\pm d$. Hence, $x=(k\pm 1)d+r'$ with $k\pm 1 \geq |r'|+1\in S$ and from minimality assumption on $|r|$, $d\in A_x$ and hence there exists two points $z_{i_1}$ and $z_{i_2}$ in the sequence $S$ such that $z_{i_1}-z_{i_2}=d$ contradicting $z_j-z_{j'}\neq d$ for $0\leq j,$ $j' \leq l$. This completes the proof of the lemma.
\end{proof}

Now we prove theorem \ref{main}.
\begin{proof}
Let $k(n,d)\geq |r(n,d)|+1$. Then $n=k(n,d)d+r(n,d)$ and the proof follows from Lemma \ref{lemma1}.
For the converse part, we prove that if $k(n,d)\leq |r(n,d)|$ then $d\notin A_n$.
For the simplicity, let $k=k(n,d)$, $r=r(n,d)$. Then $n=kd+r$. Let $h=d+1$. From now on we treat Gaussian integers as ordered pairs of integers.
Let $m\geq 0$ be an integer. Let $R_m$ and $T_m$ denote the set of Gaussian integers in the vertical line segments joining ${\big(}m(d+1),-m{\big)}$, ${\big(}m(d+1),h-m{\big)}$
and ${\big(}(m+1)(d-1), h-m{\big)}$, ${\big(}(m+1)(d-1), -m-1{\big)}$, respectively and let $S_m$ and $U_m$ denote the set of Gaussian integers in the horizontal line segment joining ${\big(}m(d+1),h-m{\big)}$, ${\big(}(m+1)(d-1),h-m{\big)}$ and ${\big(}(m+1)(d-1),-m-1{\big)}$, ${\big(}(m+1)(d+1),-m-1{\big)}$, respectively. Let $P_1$ be a set defined as follows: If $d$ is odd,
\[P_1=R_0\cup S_0\cup T_0 \cup U_0\cdots R_{\frac{d-1}{2}-1}\cup S_{\frac{d-1}{2}-1}\cup T_{\frac{d-1}{2}-1} \cup U_{\frac{d-1}{2}-1} \cup R_{\frac{d-1}{2}},\]
and if $d$ is even,
\[P_1=R_0\cup S_0\cup T_0 \cup U_0\cdots R_{\frac{d}{2}-2}\cup S_{\frac{d}{2}-2}\cup T_{\frac{d}{2}-2} \cup U_{\frac{d}{2}-2} \cup R_{\frac{d}{2}-1}\cup S_{\frac{d}{2}-1}\cup T_{\frac{d}{2}-1}.\]
Let the sets of Gaussian integers in the line segments joining ${\big(}m(d+1),-m{\big)}$, ${\big(}m(d+1),h-m{\big)}$; ${\big(}m(d+1),h-m{\big)}$, ${\big(}(m+1)(d-1)-1,h-m{\big)}$; ${\big(}(m+1)(d-1)-1,h-m{\big)}$, ${\big(}(m+1)(d-1)-1,-m-1{\big)}$ and ${\big(}(m+1)(d-1)-1,-m-1{\big)}$, ${\big(}(m+1)(d+1),-m-1{\big)}$ be $R'_m$, $S'_m$, $T'_m$ and $U'_m$, respectively.

Further, let $P_2$ be another set defined as follows. If $d$ is odd,
\[P_2=\{(-1,0)\}\cup R'_0\cup S'_0\cup T'_0 \cup U'_0\cdots R'_{\frac{d-5}{2}}\cup S'_{\frac{d-5}{2}}\cup T'_{\frac{d-5}{2}} \cup U'_{\frac{d-5}{2}} \cup R'_{\frac{d-3}{2}}\cup S'_{\frac{d-3}{2}} \cup T'_{\frac{d-3}{2}},\]
and if $d$ is even,
\[P_2=\{(-1,0)\}\cup R'_0\cup S'_0\cup T'_0 \cup U'_0\cdots R'_{\frac{d-4}{2}}\cup S'_{\frac{d-4}{2}}\cup T'_{\frac{d-4}{2}} \cup U'_{\frac{d-4}{2}} \cup R'_{\frac{d}{2}-1}.\]
It is not difficult to see that there exist two sequences $S_1$ and $S_2$ of Gaussian integers whose ranges are $P_1$ and $P_2$, respectively and such that for every two consecutive terms of sequence $z_j$ and $z_{j+1}$ of either $S_1$ or $S_2$ we have $|z_j-z_{j+1}|=1$.
In Lemma \ref{lemma2}, we prove that  neither of the sets $P_1$ or $P_2$ have two elements (picked from the same set) whose difference is $d$.

One can clearly see that if $d$ is odd, then
\[\left\{(m(d+1),0) : 0\leq m\leq \frac{(d-1)}{2} \right\} \bigcup \left\{(i(d-1),0) : 0\leq i\leq \frac{(d+1)}{2} \right\}\subset P_1,\]
\[\left\{(m(d+1),0) : 0\leq m\leq \frac{(d-3)}{2} \right\}\bigcup \left\{(i(d-1)-1,0) : 0\leq i\leq \frac{(d-1)}{2} \right\}\subset P_2.\]
and if $d$ is even, then
\[\left\{(m(d+1),0) : 0\leq m\leq \frac{d}{2}-1 \right\}\bigcup \left\{(i(d-1),0) : 0\leq i\leq \frac{d}{2} \right\}\subset P_1,\]
\[\left\{(m(d+1),0) : 0\leq m\leq \frac{d}{2}-1 \right\}\bigcup \left\{(i(d-1)-1,0) : 0\leq i\leq \frac{d}{2}-1 \right\}\subset P_2.\]

It is given that $n=kd+r$ and $k\leq |r|$. Let $k$ and $r$ are of same parity. If $r > 0$. Let  $m=\frac{r+k}{2}$ and $i=\frac{r-k}{2}$. Since $P_1$ passes through $(m(d+1),0)$ and $(i(d-1),0)$ and  $m(d+1)-i(d-1)=n$, we have $d\notin A_n$ as $P_1$ has no two elements with $d$ as a difference.

If $r < 0$, choose $m=\frac{-r-k}{2}$ and $i=\frac{-r+k}{2}$. As $(m(d+1),0)\in P_1$ and $(i(d-1),0)\in P_1$ and $i(d-1)-m(d+1)=n$, we have $d\notin A_n$.

Next, let $k$ and $r$ are of different parity. If $r > 0$, choose $m=\frac{r-1+k}{2}$ and $i=\frac{r-1-k}{2}$  and if $r < 0$, choose $m=\frac{-r-k-1}{2}$ and $i=\frac{k-r-1}{2}$.  We observe that $(m(d+1),0)\in P_2$ and $(i(d-1)-1,0)\in P_2$ and $m(d+1)-(i(d-1)-1)=\pm n$ and since $P_2$ has no two elements with $d$ as a difference, we have $d\notin A_n$. This completes the proof of the theorem.
\end{proof}

\begin{Lemma 2} \label{lemma2}
Neither of the sets $P_1$ or $P_2$ have two elements (picked from the same set) whose difference is $d$.
\end{Lemma 2}
\begin{proof}
If $d$ is odd,
\begin{align*}
P_1& =R_0\cup S_0\cup T_0 \cup U_0\cdots R_{\frac{d-1}{2}-1}\cup S_{\frac{d-1}{2}-1}\cup T_{\frac{d-1}{2}-1} \cup U_{\frac{d-1}{2}-1} \cup R_{\frac{d-1}{2}}\\
&= \cup_{i=0}^{\frac{d-1}{2}}R_i\bigcup \cup_{i=0}^{\frac{d-1}{2}-1}S_i \bigcup \cup_{i=0}^{\frac{d-1}{2}-1}T_i\bigcup \cup_{i=0}^{\frac{d-1}{2}-1}U_i,
\end{align*}
and
\begin{align*}
P_2& =\{(-1,0)\}\cup R'_0\cup S'_0\cup T'_0 \cup U'_0\cdots R'_{\frac{d-5}{2}}\cup S'_{\frac{d-5}{2}}\cup T'_{\frac{d-5}{2}} \cup U'_{\frac{d-5}{2}} \cup R'_{\frac{d-3}{2}}\cup S'_{\frac{d-3}{2}} \cup T'_{\frac{d-3}{2}}\\
&= \{(-1,0)\}\bigcup \cup_{i=0}^{\frac{d-3}{2}}R'_i\bigcup \cup_{i=0}^{\frac{d-3}{2}}S'_i\bigcup \cup_{i=0}^{\frac{d-3}{2}}T'_i\bigcup \cup_{i=0}^{\frac{d-5}{2}}U'_i.
\end{align*}

If $d$ is even,
\begin{align*}
P_1& =R_0\cup S_0\cup T_0 \cup U_0\cdots R_{\frac{d}{2}-2}\cup S_{\frac{d}{2}-2}\cup T_{\frac{d}{2}-2} \cup U_{\frac{d}{2}-2} \cup R_{\frac{d}{2}-1}\cup S_{\frac{d}{2}-1}\cup T_{\frac{d}{2}-1}\\
& = \cup_{i=0}^{\frac{d}{2}-1}R_i\bigcup \cup_{i=0}^{\frac{d}{2}-1}S_i\bigcup \cup_{i=0}^{\frac{d}{2}-1}T_i\bigcup \cup_{i=0}^{\frac{d}{2}-2}U_i,
\end{align*}
and
\begin{align*}
P_2& =\{(-1,0)\}\cup R'_0\cup S'_0\cup T'_0 \cup U'_0\cdots R'_{\frac{d-4}{2}}\cup S'_{\frac{d-4}{2}}\cup T'_{\frac{d-4}{2}} \cup U'_{\frac{d-4}{2}} \cup R'_{\frac{d}{2}-1}\\
&=\{(-1,0)\}\bigcup \cup_{i=0}^{\frac{d}{2}-1}R'_i\bigcup \cup_{i=0}^{\frac{d}{2}-2}S'_i\bigcup \cup_{i=0}^{\frac{d}{2}-2}T'_i\bigcup \cup_{i=0}^{\frac{d}{2}-2}U'_i.
\end{align*}

There are two kinds of segments, vertical: $R_i$, $T_i$, $R'_i$ and $T'_i$ and horizontal: $S_i$, $U_i$, $S'_i$ and $U'_i$.
 If $d$ is odd, then for $0\leq i \leq \frac{d-1}{2}$ and $1\leq j \leq \frac{d-1}{2}$ and if $d$ is even, then for $0\leq i \leq \frac{d}{2}-1$ and $1\leq j \leq \frac{d}{2}$, the $x$ coordinates of vertical segments of $P_1$ are $i(d+1)$ and $j(d-1)$. Hence the x-coordinates of vertical segments modulo $d$ are $i$ and $-j$ for respective intervals of $i$ and $j$ when $d$ is odd or even. Hence one can observe that these are distinct modulo $d$ and hence there cannot be any two elements differing by $d$ which belong to two vertical lines. Similarly, if $d$ is odd, then for $0\leq i \leq \frac{d-3}{2}$ and $1\leq j \leq \frac{d-1}{2}$ and if $d$ is even, then for $0\leq i \leq \frac{d}{2}-1$ and $1\leq j \leq \frac{d}{2} - 1 $, the $x$ coordinates for different vertical line segments of $P_2$ are $i(d+1)$ and $j(d-1)-1$.
Hence the $x$-coordinates of different vertical line segments are distinct modulo $d$ and hence there cannot be any two elements whose difference is $d$.
Since $(-1\pm d,0)\notin P_2$ we can ignore about $(-1,0)$ in $P_2$.
Further, $d$ cannot be achieved as a difference of any two elements of the same horizontal line as in both $P_1$ and $P_2$ the length of horizontal segments are strictly less than $d$.
As the heights of different horizontal segments do not match there cannot be any two elements differing by $d$ from any two distinct horizontal segments. The only case to be taken into consideration is  an element from a vertical segment and an element from a horizontal segment. Thus the remaining cases left to consider are points on $R_i$, $S_j$; $R_i$, $U_j$; $T_i$, $S_j$, $T_i$, $U_j$, $R'_i$, $S'_j$; $R'_i$, $U'_j$; $T'_i$, $S'_j$ and $T'_i$, $U'_j$. We show that there cannot be any two elements with $d$ as a difference in all the eight cases.

{\bf Case 1:} ($R_i$, $S_j$). In this case, to have $d$ as a difference we require $(i(d+1)\pm d,l_1)=(l_2,h-j)$, where
\begin{align*}
& 0\leq i\leq \frac{d-1}{2} \hspace{2 mm}\text{if d is odd},\\
& 0\leq i \leq \frac{d}{2}-1\hspace{2 mm}\text{if d is even},\\
& 0\leq j \leq \frac{d-1}{2}-1\hspace{2 mm}\text{if d is odd},\\
& 0\leq j \leq \frac{d}{2}-1\hspace{2 mm}\text{if d is even},
\end{align*}
and $l_1\in [-i,h-i]$, $l_2\in [j(d+1),(j+1)(d-1)]$. This implies that $h-j\in [-i,h-i]$ or $j\in [i,i+h]$. But for $j\in [i,i+h]$, $i(d+1)\pm d \notin [j(d+1),(j+1)(d-1)]$.

{\bf Case 2:} ($R_i$, $U_j$). To have $d$ as a difference we require $(i(d+1)\pm d,l_1)=(l_2,-j-1)$, where
\begin{align*}
& 0\leq i\leq \frac{d-1}{2} \hspace{2 mm}\text{if d is odd},\\
& 0\leq i \leq \frac{d}{2}-1\hspace{2 mm}\text{if d is even},\\
& 0\leq j \leq \frac{d-1}{2}-1\hspace{2 mm}\text{if d is odd},\\
& 0\leq j \leq \frac{d}{2}-2\hspace{2 mm}\text{if d is even},
\end{align*}
and  $l_1\in [-i,h-i]$ and $l_2\in [(j+1)(d-1),(j+1)(d+1)]$. This implies that $j\in [i-h-1,i-1]$ but $i(d+1)\pm d\notin [(j+1)(d-1),(j+1)(d+1)]$ for $j\leq (i-1)$.

{\bf Case 3:} ($T_i$, $S_j$). To have $d$ as a difference we require $((i+1)(d-1)\pm d,l_1)=(l_2,h-j)$, where
\begin{align*}
& 0\leq i\leq \frac{d-1}{2}-1 \hspace{2 mm}\text{if d is odd},\\
& 0\leq i \leq \frac{d}{2}-1\hspace{2 mm}\text{if d is even},\\
& 0\leq j \leq \frac{d-1}{2}-1\hspace{2 mm}\text{if d is odd},\\
& 0\leq j \leq \frac{d}{2}-1\hspace{2 mm}\text{if d is even},
\end{align*}
and  $l_1\in [-i-1,h-i]$ and $l_2\in [j(d+1),(j+1)(d-1)]$. This implies that $j\in [i,h+i+1]$ but $(i+1)(d-1)\pm d \notin [j(d+1),(j+1)(d-1)]$ for $j\in [i,h+i+1]$.

{\bf Case 4:} ($T_i$, $U_j$). To have $d$ as a difference we require $((i+1)(d-1)\pm d,l_1)=(l_2,-j-1)$, where
\begin{align*}
& 0\leq i\leq \frac{d-1}{2}-1 \hspace{2 mm}\text{if d is odd},\\
& 0\leq i \leq \frac{d}{2}-1\hspace{2 mm}\text{if d is even},\\
& 0\leq j \leq \frac{d-1}{2}-1\hspace{2 mm}\text{if d is odd},\\
& 0\leq j \leq \frac{d}{2}-2\hspace{2 mm}\text{if d is even},
\end{align*}
and  $l_1\in [-i-1,h-i]$ and $l_2\in[(j+1)(d-1),(j+1)(d+1)]$. This implies that $j\in [i-h-1,i]$ but $(i+1)(d-1)\pm d \notin [(j+1)(d-1),)(j+1)(d+1)]$.

{\bf Case 5:} ($R'_i$, $S'_j$). To have $d$ as a difference we require $(i(d+1)\pm d,l_1)=(l_2,h-j)$, where
\begin{align*}
& 0\leq i\leq \frac{d-3}{2}\hspace{2 mm}\text{If $d$ is odd },\\
&  0\leq i\leq \frac{d}{2}-1\hspace{2 mm}\text{If $d$ is even },\\
& 0\leq j\leq \frac{d-3}{2}\hspace{2 mm}\text{If $d$ is odd },\\
& 0\leq j\leq \frac{d}{2}-2\hspace{2 mm}\text{If $d$ is even },
\end{align*}
and  $l_1\in[-i,h-i]$ and $l_2\in[j(d+1),(j+1)(d-1)-1]$. This implies that $j\in [i,h+i]$ but $i(d+1)\pm d\notin [j(d+1),(j+1)(d-1)-1]$ for $j\in [i,h+i]$.

{\bf Case 6:} ($R'_i$, $U'_j$). To have $d$ as a difference we require $(i(d+1)\pm d,l_1)=(l_2,-j-1)$, where
\begin{align*}
& 0\leq i\leq \frac{d-3}{2}\hspace{2 mm}\text{If $d$ is odd },\\
&  0\leq i\leq \frac{d}{2}-1\hspace{2 mm}\text{If $d$ is even },\\
& 0\leq j\leq \frac{d-5}{2}\hspace{2 mm}\text{If $d$ is odd },\\
&  0\leq j\leq \frac{d}{2}-2\hspace{2 mm}\text{If $d$ is even },\\
\end{align*}
and  $l_1\in [-i,h-i]$ and $l_2\in [(j+1)(d-1)-1,(j+1)(d+1)]$. This implies that $j\in [i-h-1,i-1]$ but $i(d+1)\pm d\notin [(j+1)(d-1)-1,(j+1)(d+1)]$ for $j\leq i-1$.

{\bf Case 7:} ($T'_i$, $S'_j$) To have $d$ as a difference we require $((i+1)(d-1)-1\pm d,l_1)=(l_2,h-j)$, where
\begin{align*}
& 0\leq i\leq \frac{d-3}{2}\hspace{2 mm}\text{If $d$ is odd },\\
&  0\leq i\leq \frac{d}{2}-2\hspace{2 mm}\text{If $d$ is even },\\
& 0\leq j\leq \frac{d-3}{2}\hspace{2 mm}\text{If $d$ is odd },\\
&  0\leq j\leq \frac{d}{2}-2\hspace{2 mm}\text{If $d$ is even },
\end{align*}
and  $l_1\in [-i-1,h-i]$ and $l_2\in [j(d+1),(j+1)(d-1)-1]$, This implies that $j\in [i,h+i+1]$ but $(i+1)(d-1)-1\pm d\notin [j(d+1),(j+1)(d-1)-1]$ for $j\geq i$.

{\bf Case 8:} ($T'_i$, $U'_j$). To have $d$ as a difference we require $((i+1)(d-1)-1\pm d,l_1)=(l_2,-j-1)$, where
\begin{align*}
& 0\leq i\leq \frac{d-3}{2}\hspace{2 mm}\text{If $d$ is odd },\\
&  0\leq i\leq \frac{d}{2}-2\hspace{2 mm}\text{If $d$ is even },\\
& 0\leq j\leq \frac{d-5}{2}\hspace{2 mm}\text{If $d$ is odd },\\
&  0\leq j\leq \frac{d}{2}-2\hspace{2 mm}\text{If $d$ is even },
\end{align*}
and  $l_1\in [-i-1,h-i]$ and $l_2\in [(j+1)(d-1)-1,(j+1)(d+1)]$. This implies that $j\in [i-h-1,i]$ but for $j\in [i-h-1,i]$ $(i+1)(d-1)-1\pm d\notin [(j+1)(d-1)-1,(j+1)(d+1)]$. This completes the proof of the lemma.
\end{proof}

An immediate consequence of theorem \ref{main} is the following corollary.

\begin{Corollary}\label{corollary}
If $n\geq \frac{d^2}{2}$, then $d\in A_n$.
\end{Corollary}

\begin{proof}
Let $d$ be even. Clearly, $-\frac{d}{2}\leq r(n,d) \leq \frac{d}{2}-1$.
If $r(n,d) < 0$, then
\[k(n,d)>\frac{n}{d}\geq \frac{d}{2}\geq |r(n,d)|,\]
and if $r(n,d)\geq 0$, then
\[k(n,d)=\frac{n-r(n,d)}{d}\geq \frac{(n-\frac{d}{2}+1)}{d}>\frac{\frac{d^2}{2}-\frac{d}{2}}{d}\geq |r(n,d)|.\]

Next, let $d$ be odd. Clearly, $-\frac{d-1}{2}\leq r(n,d) \leq \frac{d-1}{2}$ and hence
\[k(n,d)=\frac{n-r(n,d)}{d}\geq \frac{n-\frac{d-1}{2}}{d}>\frac{d^2-d}{2d}\geq |r(n,d)|.\]
Hence, in each case $k(n,d) > |r(n,d)|$ and so by theorem \ref{main}, $d\in A_n$.
\end{proof}

Lemma \ref{lemma1} can be generalized in following theorem using the techniques of Lemma \ref{lemma1}.

\begin{sub}
If $S$ is a sequence of Gaussian integers such that every two consecutive terms $z_j$ and $z_{j+1}$ satisfy $|z_{j+1}-z_j|=1$ and there exists a pair $j_1$ and $j_2$ such that $z_{j_2}-z_{j_1}=n+ih$ and for a natural number $d$ if $k(n,d)\geq |r(n,d)|+|h|+1$ then there exists a pair of terms in the sequence $j_3$ and $j_4$ such that $z_{j_3}-z_{j_4}=d$.
\end{sub}
\begin{proof}

Let $h$ be the number with least absolute value such that there exists a  sequence $S = (z_j)$ with consecutive terms satisfying $|z_{j+1}-z_j|=1$, $z_0 = 0$, $z_l = n+ih$, $k(n,d)\geq |r(n,d)|+|h|+1$ , $z_r-z_s\neq d$ for $0\leq r, s \leq l$ and $z_l$ is the last term of the sequence. We have, $h\neq 0$ from theorem \ref{main}. We define a new sequence $S' = (z_j')$ such that $z'_j=z_j$ for $0\leq j \leq l$ and $z'_{l+1}=n+i(h+\theta)$ where $\theta$ is given by
\[
\theta=
\left\{
  \begin{array}{ll}
    -1, & \hbox{if $h\geq 0$;} \\
    +1, & \hbox{if $h<0$.}
  \end{array}
\right.
\]
Hence, $|h+\theta|=|h|-1$ and from the minimality assumption on $|h|$ there exists two terms $x\in S'$ and $y\in S'$ such that $x-y=\pm d$. Both of them cannot belong to $S$ as we assumed that there are no terms in $S$ with a  difference $d$. Hence without loss of generality, let $x=(n,h+\theta)$ and $y=(n\pm d,h+\theta)\in S$ and since $k(n\pm d,d)\geq |r(n,d)|+|h+\theta|+1$,  $|h+\theta|=|h|-1$ from minimality assumption on $|h|$ there exists two terms of $S$ with $d$ as difference which contradicts the assumption about $S$ that $z_r-z_s\neq d$ for $0\leq r, s \leq l$.
\end{proof}

We close this section by giving a formula for the cardinality of $A_n$.

\begin{exact}
The cardinality of $A_n$ is
\[ \lfloor \sqrt{2n} \rfloor + 2\left\lfloor \frac{n+1}{\sqrt{2n}+1} \right\rfloor-\sum_{\substack{d|n \\d<\frac{n+1}{\sqrt{2n}+1}}}1 +\theta(n),\]
where $\theta(n) = |\{d : d>\sqrt{2n} {\text { and }} \frac{n+1}{\sqrt{2n}+1}\leq k(n,d) <\sqrt{\frac{n}{2}}+\frac{1}{2}\}|$. Further, $\theta(n)\in \{0,1,2\}$.
\end{exact}
\begin{proof}
From corollary \ref{corollary}, if $d\leq \sqrt{2n}$ then $d\in A_n$. So we have to count the remaining $d>\sqrt{2n}$ and $d\in A_n$. If $d\in A_n$ and $d>\sqrt{2n}$ then $n=k(n,d)d+r(n,d)$ and $k(n,d)\geq |r(n,d)|+1$. So
\[k(n,d)=\frac{n-r(n,d)}{d}\leq \frac{n}{d}+\frac{1}{2}<\sqrt{\frac{n}{2}}+\frac{1}{2},\]
since $|r(n,d)|\leq \frac{d}{2}$.
Now for counting remaining $d$ we count number of $k < \sqrt{\frac{n}{2}}+\frac{1}{2}$ and count number of distinct $d>\sqrt{2n}$ for which $k(n,d)=k$.

{\bf Case 1:} ($k<\frac{n+1}{\sqrt{2n}+1}$ and $k\nmid n$).

There are two values of $r$, say $r_1$ and $r_2$ such that $k|(n-r)$ and $-(k-1)\leq r \leq (k-1)$. For $i\in\{1,2\}$, let $d_i=\frac{n-r_i}{k}$. We have
\[d_i=\frac{n-r_i}{k}\geq \frac{n-(k-1)}{k}>\sqrt{2n},\]
and
\[|r_i| \leq k-1<\frac{n+1}{\sqrt{2n}+1}-1\leq \frac{d_i-1} {2}.\]
This implies that $r_i\in [-\lfloor\frac{d}{2}\rfloor,\lceil \frac{d}{2}-1\rceil]$.
Hence $k(n,d_i)=k$ and $d_1\neq d_2$.
So each $k$ satisfying $k<\frac{n+1}{\sqrt{2n}+1}$ and $k\nmid n$ corresponds to two distinct $d>\sqrt{2n}$.

{\bf Case 2:} ($k<\frac{n+1}{\sqrt{2n}+1}$ and $k|n$).

There exists precisely one value of $r=0$ in the interval $-(k-1)\leq r \leq (k-1)$ satisfying $k|(n-r)$.
The corresponding $d$ is
\begin{equation*}
d=\frac{n}{k}>\sqrt{2n}
\end{equation*}
Hence, each such $k$ corresponds to precisely one value of $d>\sqrt{2n}$ such that $k(n,d)=k$.

{\bf Case 3:} ($\frac{n+1}{\sqrt{2n}+1} \leq k<\sqrt{\frac{n}{2}}+\frac{1}{2}$).

Since there can be atmost one $k$ in the interval and such a $k$ can at most correspond to $2$ distinct values of $d>\sqrt{2n}$. Let $\theta(n)$ correspond to the number of $d>\sqrt{2n}$ such that $\frac{n+1}{\sqrt{2n}+1}\leq k(n,d) <\sqrt{\frac{n}{2}}+\frac{1}{2}$. Clearly $\theta(n)\in\{0,1,2\}$.

We claim that distinct $k_1<\sqrt{\frac{n}{2}}+\frac{1}{2}$ and $k_2<\sqrt{\frac{n}{2}}+\frac{1}{2}$ correspond to distinct $d>\sqrt{2n}$.  Let $k(n,d_1)=k_1$ and $k(n,d_2)=k_2$. If $d_1=d_2$, then
\begin{equation}
\frac{n-r(n,d_1)}{k_1}=\frac{n-r(n,d_2)}{k_2}\implies (k_2-k_1)n=k_2r(n,d_1)-k_1r(n,d_2).
\end{equation}
Since $|r(n,d_i)|\leq k_i-1$ for $i\in\{1,2\}$, therefore
\[|k_2r(n,d_1)-k_1r(n,d_2)|\leq k_2(k_1-1)+k_1(k_2-1)<2\left(\sqrt{\frac{n}{2}}+\frac{1}{2}\right)\left(\sqrt{\frac{n}{2}}-\frac{1}{2}\right)<n.\]
Now, as the right hand side of (2.1) is a multiple of $n$ and the absolute value  is strictly less than $n$, right hand side has to be $0$ which implies $k_1=k_2$.

Hence, the count of total number of $d$ in all cases is
\begin{align*}
|A_n|=& |\{d\in \mathbb{N}:d\leq \sqrt{2n}\}|+2|\{k\in \mathbb{N}: k<\frac{n+1}{\sqrt{2n}+1} \text{ and }k\nmid n\}|+\\ & |\{k\in \mathbb{N}: k<\frac{n+1}{\sqrt{2n}+1} \text{ and }k|n\}|+\theta(n)\\
& =\lfloor \sqrt{2n} \rfloor+2\left\lfloor \frac{n+1}{\sqrt{2n}+1} \right\rfloor-\sum_{\substack{d|n \\d<\frac{n+1}{\sqrt{2n}+1}}}1 +\theta(n).
\end{align*}
\end{proof}

\begin{Coro}
The cardinality of $A_n$ for all $\epsilon>0$ is
\begin{equation*}
|A_n|=2\sqrt{2n}+O(n^\epsilon).
\end{equation*}
\end{Coro}

\section{Answers to some questions raised in \cite{F}}

Ledoan and Zaharescu (\cite{F}, section 3) raised six questions. We answer four of the six questions below.

{\bf Question 1} asks which positive integers belong to $A_n$ and theorem \ref{main} answers the question.
\par
{\bf Question 2} asks for which positive integers $n$,  $A_n$ is equal to the set of all divisors of $n$. We claim that the only numbers $n$ for which $A_n$ is equal to the set of all divisors of $n$ are $1,2,4,6,12$. One can check that $n=1,2,4,6,12$ are the only numbers $\leq 13$ such that $A_n$ is equal to the set of divisors of $n$. From corollary \ref{corollary}, both $\lfloor\sqrt{2n}\rfloor$ and  $\lfloor\sqrt{2n}-1\rfloor$ are in $A_n$ and for them to be divisors of $n$,  ($\lfloor\sqrt{2n}\rfloor)(\lfloor\sqrt{2n}-1\rfloor)|n$. But for $n\geq 14$, $(\sqrt{2n}-1)(\sqrt{2n}-2)>n$ and hence $n=1,2,4,6,12$ are the only numbers.

{\bf Question 4} For which numbers $n$ there exists a sequence of Gaussian integers $S$ such that $n\in \mathcal{A}$ (which is the set of differences of the terms of the sequence $S$) and such that for each divisor $d$ of $n$, with $1 < d < n$, either both $d-1$ and $d + 1$ are divisors of $n$, or at least one of
$d-1$ or $d + 1$ is not in $\mathcal{A}$?

We claim that the numbers which satisfy the hypothesis are precisely all prime numbers together with $\{1,4,6,12\}$.

For, if $n=1$ or $n$ is a prime then the hypothesis is vacuously true. If $n$ is a composite number  $\geq 14$, then since $\sqrt{2n}-1\geq \sqrt{n}$ for $n\geq 14$ there exists a divisor $d$ satisfying $1<d\leq\sqrt{2n}-1$. Let $d$ be the largest integer dividing $n$ and $\leq \sqrt{2n}-1$.

{\bf Case 1:} Atleast one of $d-1$ and $d+1$ does not divide $n$.

Then from corollary \ref{corollary}, $d-1 \in A_n$ and $d+1 \in A_n$ . Clearly, $n$ does not satisfy the hypothesis.

{\bf Case 2:} Both $d-1$, $d+1$ divide $n$. 
We have $d(d+1)|n$. Since we have assumed that $d$ is the greatest divisor $\leq \sqrt{2n}-1$ which implies $d+1>\sqrt{2n}-1$. $d(d+1)|n$ implies that $(\sqrt{2n}-1)(\sqrt{2n}-2)\leq n$ but $(\sqrt{2n}-1)(\sqrt{2n}-2)> n$ for $n\geq 14$.

Hence there are no composite $n\geq 14$ satisfying the hypothesis. One can check that for $n\leq 13$ and $n$ is not a prime precisely $1,4,6,12$ satisfy the hypothesis.

{\bf Question 5} For which numbers $n$ there exists a sequence of Gaussian integers such that $n\in \mathcal{A}$ and such that for each divisor $d$ of $n$ with $K < d < n-K$ either all the numbers $d-K, d-K + 1, \ldots, d + K$ are
divisors of $n$ or at least one of $d-K, d-K + 1, . . . , d + K$ is not in $\mathcal{A}$?
\par
We claim that the set of  numbers satisfying the hypothesis is $\{mp : 1\leq m\leq K \text{ and }p \text{ is a prime }\geq (2K+1)\}$ together with a finite set.

{\bf Case 1:}
 Let $n=mp$, where $m\leq K$ and $p\geq 2K+1$. Then $n=m(p+1)-m$, which implies that $|r(n,p+1)| \geq k(n,p+1).$ From theorem \ref{main}, $p+1\notin A_n$. Therefore, there exists a sequence $S$ which does not have any two elements whose difference is $p+1$ and contains two terms with difference $n$. Since any divisor $d>K$ of $n$ is of the form $d=d'p$ where $d'$ is divisor of $m$, therefore for each $d'p$, $d'p+d'$($d'\leq K$) is not in the difference set of $S$, as if $d'p+d'\in \mathcal{A}$ we have $p+1 \in \mathcal{A}$, which is not true. Hence the sequence $S$ satisfies the hypothesis.

{\bf Case 2:} $n\notin \{mp : 1\leq m\leq K \text{ and }p \text{ is a prime }\geq (2K+1)\}$, $n\geq (2K+1)K$, $\frac{\sqrt{2n}-K}{K}>K$, $\sqrt{2n}-K\geq \sqrt{n}$ and $(\sqrt{2n}-K)(\sqrt{2n}-K-1) > n$.

Clearly, $n$ is not a prime. Let $d$ be the greatest integer dividing $n$ and satisfying $d\leq \sqrt{2n}-K$. We claim that $d>K$. If $d\leq K$, then $n=md$ for some $m$. Let $p$ be a prime dividing $m$. Since $pd|n$ and $pd>d$, from maximality of $d$,  $pd>\sqrt{2n}-K$ which implies $p>\frac{\sqrt{2n}-K}{K}>K$. Since $n\notin \{mp : 1\leq m\leq K \text{ and }p \text{ is a prime }\geq (2K+1) \}$, $m$ is not a prime. Hence, $m$ has atleast two prime factors each of them being greater than $K$ and atleast one of them will be less than $\sqrt{n}\leq \sqrt{2n}-K$. Hence, atleast one of the prime factor of $m$ is greater than $d$ and $\leq \sqrt{2n}-K$, contradicting the maximality of $d$. Thus $d>K$.
\par
If $(d+1)|n$, then from maximality of $d$, $(d+1)>\sqrt{2n}-K$ and $d(d+1)|n$ which implies $(\sqrt{2n}-K)(\sqrt{2n}-K-1)\leq n$. But $(\sqrt{2n}-K)(\sqrt{2n}-K-1)> n$. Hence $(d+1)\nmid n$.

For all $1\leq i \leq K$, $d\pm i\leq \sqrt{2n}$ and from corollary \ref{corollary}, $d\pm i\in A_n$ and hence $n$ does not satisfy the hypothesis.

{\bf Case 3:} $n\notin \{mp : 1\leq m\leq K \text{ and }p \text{ is a prime }\geq (2K+1) \}$ and atleast one of the inequalities $n\geq (2K+1)K$, $\frac{\sqrt{2n}-K}{K}>K$, $\sqrt{2n}-K\geq \sqrt{n}$ and $(\sqrt{2n}-K)(\sqrt{2n}-K-1) > n$  is not true.
This accounts for finitely many exceptions.

\bibliographystyle{amsplain}

  \end{document}